\documentclass{amsart}
\usepackage{amsmath,amssymb,amsthm,amscd,hyperref,graphicx}

\DeclareMathOperator{\ev}{{ev}}
\DeclareMathOperator{\pt}{{pt}}

\def\bP{\mathbf{P}}
\def\bF{\mathbb{F}}

\def\bR{\mathbb{R}}

\newcommand{\proj}{{\bP}}

\newcommand{\cpx}{\mathbb{C}}

\newtheorem{prop}{Proposition}[section]

\newtheorem{lem}[prop]{Lemma}
\newtheorem{thm}[prop]{Theorem}
\newtheorem{cor}[prop]{Corollary}

\newtheorem{eg}[prop]{Example}
\email{sclau@math.cuhk.edu.hk, leung@math.cuhk.edu.hk, bswu@ims.cuhk.edu.hk}

\begin{document}
\title[A relation for Gromov-Witten invariants]{A relation for Gromov-Witten
invariants of local Calabi-Yau threefolds}
\author{Siu-Cheong Lau}
\author{Naichung Conan Leung}
\author{Baosen Wu}
\address{The Institute of Mathematical Sciences, The Chinese University of
Hong Kong, Shatin, N.T., Hong Kong}

\begin{abstract}
We compute certain \emph{open} Gromov-Witten invariants for toric Calabi-Yau
threefolds. The proof relies on a relation for ordinary Gromov-Witten
invariants for threefolds under certain birational transformation, and a
recent result of Kwokwai Chan.
\end{abstract}

\maketitle


\section{Introduction}

\label{intro}

The aim of this paper is to compute genus zero \emph{open} Gromov-Witten
invariants for toric Calabi-Yau threefolds, through a relation between
ordinary Gromov-Witten invariants of $K_S$ and $K_{S_n}$, where $S_n$ is a
blowup of $S$.

The celebrated SYZ mirror symmetry was initiated from the work of
Strominger-Yau-Zaslow \cite{SYZ}. For toric manifolds, the open
Gromov-Witten invariants which count holomorphic disks play a fundamental
role in the construction of their Landau-Ginzburg mirrors. For toric Fano
manifolds, Cho and Oh \cite{Cho-Oh} classified holomorphic disks with
boundary in Lagrangian torus fibers, and computed the mirror superpotential.
However, when the toric manifold is not Fano, the moduli of holomorphic
disks contains bubble configurations and it has a nontrivial obstruction
theory. The only known results are the computations of the mirror
superpotentials of Hirzebruch surface $\mathbb{F}_2$ by Fukaya-Oh-Ohta-Ono's 
\cite{FOOO2} using their machinery, and $\mathbb{F}_2$ and $\mathbb{F}_3$ by
Auroux's \cite{A} via wall-crossing.

Our first main result Theorem \ref{thm_main} identifies the genus zero open
Gromov-Witten invariant in terms of an ordinary Gromov-Witten invariant of
another Calabi-Yau threefold. For simplicity, we state its corollary for the
canonical bundles of toric surfaces.

\begin{thm}[Corollary of Theorem \protect\ref{thm_main}]
\label{cor_main} Let $K_S$ be the canonical bundle of a toric surface $S$
and $L$ be a Lagrangian torus fiber in $K_S$. We denote $\beta\in
\pi_2(K_S,L)$ as the disk class for a fiber of $K_{S}\to S$. Given any $%
\alpha\in H_2(S,\mathbb{Z})$, we assume that every rational curve
representing the strict transform $\alpha^{\prime}$ of $\alpha$ in $S_1$
cannot be deformed away from the zero section of $K_{S_1}\to S_1$.

Then the genus zero open Gromov-Witten invariant $n_{\alpha+\beta}$ of $%
(K_S,L)$ is equal to an ordinary Gromov-Witten invariant of $K_{S_1}$, that
is 
\begin{equation*}
n_{\alpha+\beta} = \langle 1 \rangle^{K_{S_1}}_{0,0, \alpha^{\prime}}.
\end{equation*}
\end{thm}

The assumption on the class $\alpha$ is needed in order for the
Gromov-Witten invariant of the noncompact space $K_{S_1}$ to be
well-defined. When $S_1$ is Fano, our assumption always holds true for any $%
\alpha$.

We remark that the theorem holds for the more general setting where the disk
boundary has $n$ components lying in $n$ distinct Lagrangian torus fibers.

We prove Theorem \ref{thm_main} using our second main result stated below
and Chan's result \cite{C} relating open and closed Gromov-Witten invariants.

Let $S$ be a smooth projective surface and $X$ be a fiberwise
compactification of $K_S$, i.e., $p:X=\mathbf{P}(K_S\oplus \mathcal{O}%
_{S})\to S$ as a ${\mathbf{P}^{1}}$-bundle. We relate certain $n$-point
Gromov-Witten invariants of $X$ to the Gromov-Witten invariants (with no
point condition) of $W$, the fiberwise compactification of $K_{S_n}$, where $%
S_n$ is the blowup of $S$ at $n$ points.

\begin{thm}
\label{main} For $\beta\in H_2(X,\mathbb{Z})$ whose intersection number with
the infinity section of $X$ is $n$. Let $\beta^{\prime}\in H_2(S_n,\mathbb{Z}%
)$ be the strict transform of $p_*(\beta)\in H_2(S,\mathbb{Z})$. Then 
\begin{equation*}
\langle[\pt],\cdots,[\pt]\rangle^X_{0,n,\beta}=\langle
1\rangle^W_{0,0,\beta^{\prime}}. 
\end{equation*}
Here $[\pt]$ is the Poincar\'e dual of the point class.
\end{thm}

Now we sketch the proof of Theorem \ref{main} in the case $n=1$, that is, 
\begin{equation*}
\langle[\pt]\rangle^X_{0,1,\beta}=\langle 1\rangle^W_{0,0,\beta^{\prime}}
\end{equation*}
for $\beta=\alpha+h$, and $\beta^{\prime}=\pi^!\alpha-e$, the strict
transform of $\alpha$.

Recall that $X=\mathbf{P}(K_S\oplus \mathcal{O}_{S})$ and $W=\mathbf{P}%
(K_{S_1}\oplus \mathcal{O}_{S_1})$ with $\pi:S_1\to S$ the blowup of $S$ at
one point with exceptional divisor $e$. Fixing a generic fiber $H$ of $X$,
let $x$ be the intersection point of $H$ with the divisor at infinity of $X$%
. We construct a birational map $f:X\overset{\pi_1}{\longleftarrow }{\tilde X%
}\overset{\pi_2}{\dashrightarrow }W$ so that $\pi_1$ is the blow up at $x$,
and $\pi_2$ is a simple flop along ${\tilde H}$, the proper image of $H$
under $\pi_1$. We compare Gromov-Witten invariants of $X$ and $W$ through
the intermediate space ${\tilde X}$. The identity follows from the results
of Gromov-Witten invariants under birational transformations which are
listed in Section \ref{Review}.

We remark that Theorem \ref{main} is a corollary of Proposition \ref%
{GWblowup} which holds for all genera. They can be generalized to the case
when $K_S$ is replaced by other local Calabi-Yau threefolds, as we shall
explain in Section \ref{bundle}. \vskip 5pt

This paper is organized as follows. Section \ref{Review} serves as a brief
review on definitions and results that we need in Gromov-Witten theory. In
Section \ref{bundle} we prove Theorem \ref{main} and its generalization to
quasi-projective threefolds. In Section \ref{toric} we deal with toric
Calabi-Yau threefolds and prove Theorem \ref{thm_main}. Finally in Section %
\ref{generalization} we generalize Theorem \ref{main} to ${\mathbf{P}^{n}}$%
-bundles over an arbitrary smooth projective variety.

\textbf{Acknowledgements.} We thank Kwokwai Chan for the stimulating
discussions and his preprint \cite{C} on the comparison of Kuranishi
structures. His ideas on the relationship between open Gromov-Witten
invariants and mirror periods inspired our work. The first author is very
grateful to Mark Gross for the enlightening discussions on wall-crossing and
periods. We also thank Jianxun Hu for helpful comments. The authors are
partially supported by RGC grants from the Hong Kong Government.

\section{Gromov-Witten invariants under birational maps}

\label{Review}

In this section we review Gromov-Witten invariants and their transformation
under birational maps.

Let $X$ be a smooth projective variety. Let $\overline M_{g,n}(X,\beta)$ be
the moduli space of stable maps $f:(C; x_1,\cdots x_n)\to X$ with genus $%
g(C)=g$ and $[f(C)]=\beta\in H_2(X,\mathbb{Z})$. Let $\ev_i: \overline
M_{g,n}(X,\beta)\to X$ be the evaluation map $f\mapsto f(x_i)$. The
Gromov-Witten invariant for classes $\gamma_i\in H^*(X)$ is defined as 
\begin{equation*}
\langle\gamma_1,\cdots,\gamma_n\rangle^X_{g,n,\beta}=\int_{[\overline
M_{g,n}(X,\beta)]^{^{\mathrm{vir}}}} \prod_{i=1}^n\ev_i^*(\gamma_i).
\end{equation*}

When the expected dimension of $\overline M_{g,n}(X,\beta)$ is zero, for
instance, $X$ is a Calabi-Yau threefold, we will be interested primarily in
the invariant 
\begin{equation*}
\langle 1\rangle^X_{g,0,\beta}=\int_{[\overline M_{g,n}(X,\beta)]^{^{\mathrm{%
vir}}}}1
\end{equation*}
which equals to the degree of the $0$-cycle $[\overline M_{g,n}(X,\beta)]^{%
\mathrm{vir}}$.

Roughly speaking, the invariant $\langle\gamma_1,\cdots,\gamma_n%
\rangle^X_{g,n,\beta}$ is a virtual count of genus $g$ curves in the class $%
\beta$ which intersect with generic representatives of the Poincar\'e dual $%
PD(\gamma_i)$. In particular, if we want to count curves in a homology class 
$\beta$ passing through a generic point $x\in X$, we simply take some $%
\gamma_i$ to be the cohomology class $[\pt]$ of a point. There is an
alternative way to do this counting: let $\pi:{\tilde X}\to X$ be the blow
up of $X$ along $x$; we count curves in the homology class $\pi^!(\beta)-e$,
where $\pi^!(\beta)=PD(\pi^*PD(\beta))$ and $e$ is the class of a line in
the exceptional divisor. The following result says that for genus zero case,
these two methods give the same result.

\begin{thm}
\label{thmGH}(\cite{Ga},\cite{Hu}) Let $\pi:{\tilde X}\to X$ be the blowup
at a point. Let $e$ be the line class in the exceptional divisor. Let $%
\beta\in H_2(X),\gamma_1,\cdots,\gamma_n\in H^*(X)$. Then 
\begin{equation*}
\langle\gamma_1,\cdots,\gamma_n,[\pt]\rangle^X_{0,n+1,\beta}=
\langle\pi^*\gamma_1,\cdots,\pi^*\gamma_n\rangle^{{\tilde X}}_{0,n,{%
\pi^!(\beta)-e}}.
\end{equation*}
\end{thm}

Another result that we need concerns the transformation of Gromov-Witten
invariants under flops.

Let $f: X\dashrightarrow X_f$ be a simple flop along a smooth $(-1,-1)$
rational curve between two threefolds. There is a natural isomorphism 
\begin{equation*}
\varphi: H_2(X,\mathbb{Z})\to H_2(X_f,\mathbb{Z}).
\end{equation*}
Suppose that $\Gamma$ is an exceptional curve on $X$ and $\Gamma_{\!f}$ is
the corresponding exceptional curve on $X_f$. Then 
\begin{equation*}
\varphi([\Gamma])=-[\Gamma_{\!f}].
\end{equation*}

The following theorem is proved by A.-M.~Li and Y.~Ruan.

\begin{thm}
\label{thmLR}(\cite{LR}) For a simple flop $f:X\dashrightarrow X_f$, if $%
\beta\ne m[\Gamma]$ for any exceptional curve $\Gamma$, we have 
\begin{equation*}
\langle\varphi^*\gamma_1,\cdots,\varphi^*\gamma_n\rangle^X_{g,n,\beta}=%
\langle\gamma_1,\cdots,\gamma_n\rangle^{X_f}_{g,n,\varphi(\beta)}.
\end{equation*}
\end{thm}

\section{Gromov-Witten invariants of projectivization of $K_{S}$}

\label{bundle}

In this section we prove Theorem \ref{main} and its generalization to
certain quasi-projective threefolds.

To begin with, we recall some notations.

Let $S$ be a smooth projective surface. Let $p:X=\mathbf{P}(K_S\oplus 
\mathcal{O}_{S})\to S$ be a ${\mathbf{P}^{1}}$-bundle. $S$ is contained in $X
$ as the zero section of the bundle $K_S$. Denote by $S^+$ the section at
infinity. Let $h$ be the fiber class of $X$. Then any $\beta\in H_2(X,%
\mathbb{Z})$ can be written as $\alpha+nh$ for a class $\alpha$ in $H_2(S,%
\mathbb{Z})$. Here $n$ is the intersection number of $\beta$ with the
infinity section of $X$, and $p_*(\beta)=\alpha$. By Riemann-Roch, the
expected dimension of $\overline M_{0,n}(X,\beta)$ is $3n$. We have the
Gromov-Witten invariant 
\begin{equation*}
\langle[\pt],\cdots,[\pt]\rangle_{0,n,\beta}^X
\end{equation*}
which counts rational curves in the class $\beta$ passing through $n$
generic points.

Let $x_1,\cdots,x_n$ be $n$ distinct points in $X$. Let $y_i=p(x_i)\in S$.
Let $\pi:S_n\to S$ be the blowup of $S$ along the set of points $%
y_1,\cdots,y_n$ with exceptional divisors $e_1,\cdots,e_n$. We form $%
\beta^{\prime}=\pi^!\alpha-\sum_{i=1}^n e_i\in H_2(S_n,\mathbb{Z})$, which
is called the strict transform of $\alpha$. Denote $W=\mathbf{P}%
(K_{S_n}\oplus \mathcal{O}_{S_n})$. Then $\beta^{\prime}$ is a homology
class of $W$ since $S_n\subset W$. The moduli space $\overline
M_{0,0}(W,\beta^{\prime})$ has expected dimension zero, we have the
Gromov-Witten invariant $\langle 1\rangle^W_{0,0,\beta^{\prime}}$.

\begin{prop}
\label{GWblowup} Let $S$ be a smooth projective surface. Denote $p:X=\mathbf{%
P}(K_S\oplus \mathcal{O}_{S})\to S$ . Let $X_1$ be the blowup of $X$ at a
point $x$ on the infinity section of $X\to S$. Let $W=\mathbf{P}%
(K_{S_1}\oplus\mathcal{O}_{S_1})$ where $\pi:S_1\to S$ is the blowup of $S$
at the point $y=p(x)$. Then $W$ is a simple flop of $X_1$ along the proper
transform ${\tilde H}$ of the fiber $H$ through $x$.
\end{prop}

\begin{proof}
Since ${\tilde H}$ is the proper transform of $H$ under the blowup $%
\pi_1:X_1\to X$ at $x$, ${\tilde H}$ is isomorphic to ${\mathbf{P}^{1}}$
with normal bundle $\mathcal{O}(-1)\oplus\mathcal{O}(-1)$. We have a simple
flop $f:X_1\dashrightarrow X^{\prime}$ along ${\tilde H}$. Next we show that 
$X^{\prime}\cong W$. To this end, we use an alternative way to describe the
birational map $f\pi_1^{-1}:X\dashrightarrow X^{\prime}$.

It is well known that a simple flop $f$ is a composite of a blowup and a
blowdown. Let $\pi_2:X_2\to X_1$ be the blowup of $X_1$ along ${\tilde H}$
with exceptional divisor $E_2\cong {\tilde H}\times {\mathbf{P}^{1}}$.
Because the restriction of normal bundle of $E_2$ to ${\tilde H}$ is $%
\mathcal{O}(-1)$, we can blow down $X_2$ along the ${\tilde H}$ fiber
direction of $E_2$ to get $\pi_3:X_2\to X^{\prime}$. Of course we have $%
f=\pi_3\pi_2^{-1}$ and $\pi_3\pi_2^{-1}\pi_1^{-1}:X\dashrightarrow X^{\prime}
$.

Notice that the composite $\pi_2^{-1}\pi_1^{-1}:X\dashrightarrow X_2$ can be
written in another way. Let $\rho_1:Z_1\to X$ be the blowup of $X$ along $H$
with exceptional divisor $E^{\prime}$. Let $F$ be the inverse image $%
\rho^{-1}(x)$. Then $F\cong {\mathbf{P}^{1}}$. Next we blow up $Z_1$ along $F
$ to get $\rho_2:Z_2\to Z_1$. It is straightforward to verify that $Z_2=X_2$
and $\rho_1\rho_2=\pi_1\pi_2$. Thus we have $\pi_3\pi_2^{-1}\pi_1^{-1}=%
\pi_3(\rho_1\rho_2)^{-1}:X\dashrightarrow X^{\prime}$, from which it follows
easily that $X^{\prime}\cong W$.
\end{proof}

\begin{cor}
With notations as in the Proposition, and let $e_{1}$ be the exceptional
curve class of $\pi $, we have 
\begin{equation*}
\langle 1\rangle _{g,0,\beta }^{X_{1}}=\langle 1\rangle _{g,0,\beta ^{\prime
}}^{W}
\end{equation*}%
where $\beta =\alpha +k{\tilde{H}}$ and $\beta ^{\prime }=\pi ^{!}\alpha
-ke_{1}$ for any nonzero $\alpha \in H_{2}(S,\mathbb{Z})$.
\end{cor}

\begin{proof}
>From the Proposition, we know there is a flop $f:X_1\dashrightarrow W$.
Applying Theorem \ref{thmLR} to the flop $f$, since $\varphi([{\tilde H}%
])=-e_1$, we get 
\begin{equation*}
\varphi(\beta)=\varphi((\pi_1^!\alpha)+[k{\tilde H}])=\pi^!\alpha-ke_1=%
\beta^{\prime}.
\end{equation*}
It follows that 
\begin{equation*}
\langle 1\rangle^{X_1}_{g,0,\beta}=\langle 1\rangle^{W}_{g,0,\beta^{\prime}}.
\end{equation*}
\end{proof}

When $S_{1}$ is a Fano surface, $K_{S_{1}}$ is a local Calabi-Yau threefold
and curves inside $S_{1}$ can not be deformed away from $S_{1}$. Indeed any
small neighborhood $N_{S_{1}}$ of $S_{1}$ (resp. $N_{S\cup C}$ of $S\cup C$)
inside any Calabi-Yau threefold has the same property. Here $C$ is a $\left(
-1,-1\right) $-curve which intersects $S$ transversely at a single point.
Therefore we can define local Gromov-Witten invariants for $N_{S_{1}}$ and $%
N_{S\cup C}$. Using a canonical identification,%
\begin{equation*}
H_{2}\left( S_{1}\right) \simeq H_{2}\left( S\right) \oplus \mathbb{Z}%
\left\langle e_{1}\right\rangle \simeq H_{2}\left( S\cup C\right) \text{,}
\end{equation*}%
the above corollary implies that the local Gromov-Witten invariants for
local Calabi-Yau threefolds $N_{S_{1}}$ and $N_{S\cup C}$ are the same. When
the homology class in $S_{1}$ does not have $e_{1}$-component, this becomes
simply the local Gromov-Witten invariants for $N_{S}$. This last relation
for Gromov-Witten invariants of $N_{S_{1}}$ and $N_{S}$ was pointed out to
us by J.~Hu \cite{Hu2} and he proved this result by the degeneration method.
This relationship was first observed by Chiang-Klemm-Yau-Zaslow \cite{CKYZ}
in the case $S$ is $\mathbf{P}^{2}$ and genus is zero by explicit
calculations.

These results can be generalized to the case when $K_{S}$ is replaced by
other local Calabi-Yau threefolds. The illustration of such a generalization
is given at the end of this section. \newline

Now we prove Theorem \ref{main}, that is 
\begin{equation*}
\langle[\pt],\cdots,[\pt]\rangle^X_{0,n,\beta}=\langle
1\rangle^W_{0,0,\beta^{\prime}}.
\end{equation*}

\begin{proof}[Proof of Theorem \protect\ref{main}]
First we assume $n=1$, that is, $\pi:S_1\to S$ is a blowup of $S$ at one
point $y$ with exceptional curve class $e_1$ and $W=\mathbf{P}(K_{S_1}\oplus%
\mathcal{O}_{S_1})$. We need to show that 
\begin{equation*}
\langle[\pt]\rangle^X_{0,1,\beta}=\langle 1\rangle^W_{0,0,\beta^{\prime}},
\end{equation*}
where $\beta=\alpha+h$ and $\beta^{\prime}=\pi^!\alpha-e_1$.

Applying Theorem \ref{thmGH} to $\pi_1:X_1\to X$, and notice that 
\begin{equation*}
\pi_1^!(\beta)-e=\pi_1^!(\alpha+h)-e=\pi_1^!\alpha+[{\tilde H}],
\end{equation*}
which we denote by $\beta_1$, we then have $\langle[\pt]\rangle^X_{0,1,%
\beta}=\langle 1\rangle^{X_1}_{0,0,\beta_1}$. Next we apply Proposition \ref%
{GWblowup} for $k=1$, we get 
\begin{equation*}
\langle 1\rangle^{X_1}_{0,0,\beta_1}=\langle
1\rangle^{W}_{0,0,\beta^{\prime}},
\end{equation*}
which proves the result for $n=1$.

For $n>1$, we simply apply the above procedure successively.
\end{proof}

In particular, when $S={\mathbf{P}^{2}}$ and $n=1$, $S_1$ is the Hirzebruch
surface $\mathbb{F}_1$. We use $\ell$ to denote the line class of ${\mathbf{P%
}^{2}}$. The class of exceptional curve $e$ represents the unique minus one
curve in $\mathbb{F}_1$ and $f=\pi^!\ell-e$ is its fiber class. In this
case, the corresponding class $\beta^{\prime}=k\pi^!\ell-e=(k-1)e+kf$. The
values of $N_{0,\beta^{\prime}}$ have been computed in \cite{CKYZ}. Starting
with $k=1$, they are $-2, 5, -32, 286, -3038, 35870$. (See Table \ref%
{tableF1}.)

\begin{table}[h!]
\begin{center}
\begin{tabular}{|c|rrrrrrrr|}
\hline
& b & 0 & 1 & 2 & 3 & 4 & 5 & 6 \\ \hline
a &  &  &  &  &  &  &  &  \\ 
0 &  &  & $-2$ & 0 & 0 & 0 & 0 & 0 \\ 
1 &  & 1 & 3 & 5 & 7 & 9 & 11 & 13 \\ 
2 &  & 0 & 0 & $-6$ & $-32$ & $-110$ & $-288$ & $-644$ \\ 
3 &  & 0 & 0 & 0 & 27 & 286 & 1651 & 6885 \\ 
4 &  & 0 & 0 & 0 & 0 & $-192$ & $-3038$ & $-25216$ \\ 
5 &  & 0 & 0 & 0 & 0 & 0 & 1695 & 35870 \\ 
6 &  & 0 & 0 & 0 & 0 & 0 & 0 & $-17064$ \\ \hline
\end{tabular}%
\end{center}
\par
\vskip 5pt
\caption{Invariants of $K_{\mathbb{F}_1}$ for classes $ae+bf$}
\label{tableF1}
\end{table}

Next we generalize Theorem \ref{main} to quasi-projective threefolds.

Let $X$ be a smooth quasi-projective threefold. Assume there is a
distinguished Zariski open subset $U\subset X$, so that $U$ is isomorphic to
the canonical line bundle $K_S$ over a smooth projective surface $S$, and
there is a Zariski open subset $S^{\prime}\subset S$, so that each fiber $F$
of $K_S$ over $S^{\prime}$ is closed in $X$. Typical examples of such
threefolds include a large class of toric Calabi-Yau threefolds.

Theorem \ref{main} still holds for such threefolds with some mild condition.
Now we sketch the proof.

First we construct a partial compactification $\bar X$ of $X$. Given a
generic point $x\in U$, we have a unique fiber through $x$, say $H$. Let $%
\{y\}=H\cap S$. Take a small open neighborhood $y\in V$, we compactify $K_V$
along the fiber by adding a section at infinity as we did before. We call
the resulting variety by $\bar X$.

The Gromov-Witten invariant $\langle[\pt]\rangle_{0,1,\beta}^{\bar X}$ is
well defined. Indeed, let $\beta\in H_2(\bar X,\mathbb{Z})$; and suppose $%
\beta=\alpha+[H]$ for some $\alpha$ in $H_2(S,\mathbb{Z})$. The moduli space
of genus zero stable maps to $\bar X$ representing $\beta$ and passing
through the generic point $x$ is compact, provided that $S$ is Fano. Then
the invariants can be defined as before.

To show the equality $\langle[\pt]\rangle_{0,1,\beta}^{\bar X}=\langle
1\rangle_{0,0,\beta^{\prime}}^{{\tilde S}}$, we construct a birational map $%
f:\bar X\dashrightarrow W$ as in the proof of Theorem \ref{main}. Let ${%
\tilde S}\subset W$ be the image of $S$. Then ${\tilde S}$ is the blowup of $%
S$ at $y$. Let $\beta^{\prime}\in H_2({\tilde S},\mathbb{Z})$ be the strict
transform of $\alpha$. Suppose that every rational curve in $\beta^{\prime}$
lies in ${\tilde S}$, for instance, when ${\tilde S}$ is Fano, or ${\tilde S}
$ can be contracted by a birational morphism, then we can define local
Gromov-Witten invariant $\langle 1\rangle_{0,0,\beta^{\prime}}^{{\tilde S}}$%
. The equality follows directly as in the proof of Theorem \ref{main}.

\section{Toric Calabi-Yau threefolds}

\label{toric}

In this section we prove our main Theorem \ref{thm_main}. As an application,
we show that certain open Gromov-Witten invariants for toric Calabi-Yau
threefolds can be computed via local mirror symmetry.

First we recall the standard notations. Let $N$ be a lattice of rank $3$, $M$
be its dual lattice, and $\Sigma_0$ be a strongly convex simplicial fan
supported in $N_\bR$, giving rise to a toric variety $X_0 = X_{\Sigma_0}$. ($%
\Sigma_0$ is `strongly convex' means that its support $|\Sigma_0|$ is convex
and does not contain a whole line through the origin.) Denote by $v_i \in N$
the primitive generators of rays of $\Sigma_0$, and denote by $D_i$ the
corresponding toric divisors for $i = 0, \ldots, m-1$, where $m \in \mathbb{Z%
}_{\geq 3}$ is the number of such generators. \vskip 5pt

\noindent\textbf{Calabi-Yau condition for $X_0$:} There exists $\underline{%
\nu} \in M$ such that $\left( \underline{\nu}\, , \, v_i \right) = 1 $ for
all $i = 0, \ldots, m-1$.

By fixing a toric Kaehler form $\omega$ on $X_0$, we have a moment map $\mu:
X_0 \to P_0$, where $P_0 \subset M_\bR$ is a polyhedral set defined by a
system of inequalities 
\begin{equation*}
\left( v_j\, , \, \cdot \right) \geq c_j
\end{equation*}
for $j = 0, \ldots, m-1$ and suitable constants $c_j \in \mathbb{R}$.
(Figure \ref{KP2_KF1} shows two examples of toric Calabi-Yau varieties.)%
\newline

To investigate genus zero open Gromov-Witten invariants of $X_0$, we start
with the following simple lemma for rational curves in toric varieties:

\begin{lem}
\label{hol_sphere} Let $Y$ be a toric variety which admits $\nu \in M$ such
that $\nu$ defines a holomorphic function on $Y$ whose zeros contain all
toric divisors of $Y$. Then the image of any non-constant holomorphic map $%
u: {\mathbf{P}}^1 \to Y$ lies in the toric divisors of $Y$. In particular
this holds for a toric Calabi-Yau variety.
\end{lem}

\begin{proof}
Denote the holomorphic function corresponding to $\nu \in M$ by $f$. Then $f
\circ u$ gives a holomorphic function on ${\mathbf{P}}^1$, which must be a
constant by maximal principle. $f \circ u$ cannot be constantly non-zero, or
otherwise the image of $u$ lies in $(\mathbb{C}^\times)^n \subset Y$,
forcing $u$ to be constant. Thus $f \circ u \equiv 0$, implying the image of 
$u$ lies in the toric divisors of $Y$.

For a toric Calabi-Yau variety $X_0$, $\left( \underline{\nu}\, , \, v_i
\right) = 1 > 0$ for all $i = 0, \ldots, m-1$ implies that the meromorphic
function corresponding to $\underline{\nu}$ indeed has no poles.
\end{proof}

Let $L \subset X_0$ be a Lagrangian torus fiber and $b \in \pi_2(X_0, L)$ of
Maslov index two. We consider the moduli space $\overline M_{1}(X_0,b)$ of
stable maps from bordered Riemann surfaces of genus zero with one boundary
marked point to $X_0$ in the class $b$. Fukaya-Oh-Ohta-Ono \cite{FOOO}
defines the invariant 
\begin{equation*}
n_b := \int_{[\overline M_1(X_0,b)]} {\ev}^* [\pt].
\end{equation*}

We have the following

\begin{prop}
Let $\beta_i \in \pi_2(X_0, L)$ be a disc class of Maslov index two such
that $\beta_i \cdot D_j = \delta_{ij}$ for $i,j = 0, \ldots, m-1$. Then $%
\overline M_1(X_0,b)$ is empty unless $b = \beta_i$ for some $i$, or $b =
\beta_i + \alpha$, where $D_i$ is a compact toric divisor of $X_0$ and $%
\alpha \in H_2(X_0, \mathbb{Z})$ is represented by a rational curve.
\end{prop}

\begin{proof}
By \cite{FOOO}, $\overline M_1(X_0,b)$ is empty unless $b = \beta_i + \alpha$
for some $i = 0, \ldots, m-1$ and $\alpha \in H_2(X_0, \mathbb{Z})$ has
Chern number $0$. Now suppose $\overline M_1(X_0,b)$ is non-empty and $%
\alpha \not= 0$. Then $\alpha$ is realized by some chains of non-constant
holomorphic spheres $Q$ in $X_0$, which by Lemma \ref{hol_sphere} must lie
inside $\bigcup_{i=0}^{m-1} D_i$. $Q$ must have non-empty intersection with
the holomorphic disk representing $\beta_i \in \pi_2(X_0, L)$ for generic $L$%
, implying some components of $Q$ lie inside $D_i$ and have non-empty
intersection with the torus orbit $(\mathbb{C}^\times)^2 \subset D_i$. But
if $D_i$ is non-compact, then the fan of $D_i$ is simplicial convex
incomplete, and so $D_i$ is a toric manifold satisfying the condition of
Lemma \ref{hol_sphere}, forcing $Q$ to have empty intersection with $(%
\mathbb{C}^\times)^2 \subset D_i$.
\end{proof}

It was shown \cite{Cho-Oh}\cite{FOOO} that $n_b = 1$ for basic disc classes $%
b = \beta_i$. The remaining task is to compute $n_b$ for $b = \beta_i +
\alpha$ with nonzero $\alpha \in H_2(X_0)$. In this section we prove Theorem %
\ref{thm_main}, which relates $n_b$ to certain closed Gromov-Witten
invariants, which can then be computed by usual localization techniques.

Suppose we would like to compute $n_b$ for $b = \beta_i + \alpha$, and
without loss of generality let's take $i = 0$ and assume that $D_0$ is a
compact toric divisor. We construct a toric compactification $X$ of $X_0$ as
follows. Let $v_0$ be the primitive generator corresponding to $D_0$, and we
take $\Sigma$ to be the refinement of $\Sigma_0$ by adding the ray generated
by $v_{\infty} := -v_0$ (and then completing it into a convex fan). We
denote by $X = X_{\Sigma}$ the corresponding toric variety, which is a
compactification of $X_0$. We denote by $h \in H_2(X, \mathbb{Z})$ the fiber
class of $X$, which has the property that $h \cdot D_0 = h \cdot D_\infty = 1
$ and $h \cdot D = 0$ for all other irreducible toric divisors $D$. Then for 
$\alpha \in H_2(X_0, \mathbb{Z})$, we have the ordinary Gromov-Witten
invariant $\langle[\pt]\rangle^X_{0,1, h + \alpha}$.

When $X_0 = K_S$ for a toric Fano surface $S$ and $D_0$ is the zero section
of $K_S \to S$, by comparing the Kuranishi structures on moduli spaces, it
was shown by K.-W. Chan \cite{C} that the open Gromov-Witten invariant $n_b$
indeed agrees with the closed Gromov-Witten invariant $\langle[\pt]%
\rangle^X_{0,1, h + \alpha}$:

\begin{prop}[\protect\cite{C}]
\label{openclose} Let $X_0 = K_S$ for a toric Fano surface $S$ and $X$ be
the fiberwise compactification of $X_0$. Let $b= \beta_i + \alpha$ with $%
\beta_i\cdot S=1$ and $\alpha \in H_2(S, \mathbb{Z})$. Then 
\begin{equation*}
n_b = \langle[\pt]\rangle^X_{0,1, h + \alpha}.
\end{equation*}
\end{prop}

Indeed his proof extends to our setup without much modification, and for the
sake of completeness we show how it works:

\begin{prop}[slightly modified from \protect\cite{C}]
\label{openclose2} Let $X_0$ be a toric Calabi-Yau manifold and $X$ be its
compactification constructed above. Let $b= \beta_i + \alpha$ with $%
\beta_i\cdot S=1$ and $\alpha \in H_2(S, \mathbb{Z})$, and we assume that
all rational curves in $X$ representing $\alpha$ are contained in $X_0$.
Then 
\begin{equation*}
n_b = \langle[\pt]\rangle^X_{0,1, h + \alpha}.
\end{equation*}
\end{prop}

\begin{proof}
For notation simplicity let $M_{\mathrm{op}} := \overline M_1(X_0,b)$ be the
open moduli and $M_{\mathrm{cl}} := \overline M_1(X,h + \alpha)$ be the
corresponding closed moduli. By evaluation at the marked point we have a $%
\mathbf{T}$-equivariant fibration 
\begin{equation*}
\mathrm{ev}: M_{\mathrm{op}} \to \mathbf{T}
\end{equation*}
whose fiber at $p \in \mathbf{T} \subset X_0$ is denoted as $M_{\mathrm{op}%
}^{\mathrm{ev} = p}$. Similarly we have a $\mathbf{T}_\cpx$-equivariant
fibration 
\begin{equation*}
\mathrm{ev}: M_{\mathrm{cl}} \to \bar{X}
\end{equation*}
whose fiber is $M_{\mathrm{cl}}^{\mathrm{ev} = p}$.  
By the assumption that all rational curves in $X$ representing $\alpha$ is
contained in $X_0$, one has 
\begin{equation*}
M_{\mathrm{op}}^{\mathrm{ev} = p} = M_{\mathrm{cl}}^{\mathrm{ev} = p}.
\end{equation*}

There is a Kuranishi structure on $M_{\mathrm{cl}}^{\mathrm{ev} = p}$ which
is induced from that on $M_{\mathrm{cl}}$ (please refer to \cite{Fukaya-Ono}
and \cite{FOOO_book} for the definitions of Kuranishi structures).
Transversal multisections of the Kuranishi structures give the virtual
fundamental cycles $[M_{\mathrm{op}}] \in H_n (X_0, \mathbb{Q})$ and $[M_{%
\mathrm{op}}^{\mathrm{ev} = p}] \in H_0 (\{p\}, \mathbb{Q})$. In the same
way we obtain the virtual fundamental cycles $[M_{\mathrm{cl}}] \in H_{2n}
(X, \mathbb{Q})$ and $[M_{\mathrm{cl}}^{\mathrm{ev} = p}] \in H_0 (\{p\}, 
\mathbb{Q})$. By taking the multisections to be $\mathbf{T}_\cpx$- ($\mathbf{%
T}$-) equivariant so that their zero sets are $\mathbf{T}_\cpx$- ($\mathbf{T}
$-) invariant, 
\begin{equation*}
\deg [\overline M_{\mathrm{cl/op}}^{\mathrm{ev} = p}] = \deg [\overline M_{%
\mathrm{cl/op}}]
\end{equation*}
and thus it remains to prove that the Kuranishi structures on $M_{\mathrm{cl}%
}^{\mathrm{ev} = p}$ and $M_{\mathrm{op}}^{\mathrm{ev} = p}$ are the same.

Let $[u_{\mathrm{cl}}] \in M_{\mathrm{cl}}^{\mathrm{ev} = p}$, which
corresponds to an element $[u_{\mathrm{op}}] \in M_{\mathrm{op}}^{\mathrm{ev}
= p}$. $u_{\mathrm{cl}}: (\Sigma,q) \to X$ is a stable holomorphic map with $%
u_{\mathrm{cl}} (q) = p$. $\Sigma$ can be decomposed as $\Sigma_0 \cup
\Sigma_1$, where $\Sigma_0 \cong {\mathbf{P}}^1$ such that $u_* [\Sigma_0]$
represents $h$, and $u_* [\Sigma_1]$ represents $\alpha$. Similarly the
domain of $u_{\mathrm{op}}$ can be docomposed as $\Delta \cup \Sigma_1$,
where $\Delta \subset \mathbb{C}$ is the closed unit disk.

We have the Kuranishi chart $(V_{\mathrm{cl}},E_{\mathrm{cl}},\Gamma_{%
\mathrm{cl}},\psi_{\mathrm{cl}},s_{\mathrm{cl}})$ around $u_{\mathrm{cl}}
\in M_{\mathrm{cl}}^{\mathrm{ev} = p}$, where we recall that $E_{\mathrm{cl}%
} \oplus \mathrm{Im} (D_{u_{\mathrm{cl}}} \bar{\partial}) = \Omega^{(0,1)}
(\Sigma, u_{\mathrm{cl}}^* TX)$ and $V_{\mathrm{cl}} = \{\bar{\partial} f
\in E; f(q) = p\}$. On the other hand let $(V_{\mathrm{op}},E_{\mathrm{op}%
},\Gamma_{\mathrm{op}},\psi_{\mathrm{op}},s_{\mathrm{op}})$ be the Kuranishi
chart around $u_{\mathrm{op}} \in M_{\mathrm{op}}^{\mathrm{ev} = p}$.

Now comes the key: since the obstruction space for the deformation of $u_{%
\mathrm{cl}} |_{\Sigma_0}$ is $0$, $E_{\mathrm{cl}}$ is of the form $0
\oplus E^{\prime}\subset \Omega^{(0,1)} (\Sigma_0, u_{\mathrm{cl}%
}|_{\Sigma_0}^* TX) \times \Omega^{(0,1)} (\Sigma_1, u_{\mathrm{cl}%
}|_{\Sigma_1}^* TX)$. Similarly $E_{\mathrm{op}}$ is of the form $0 \oplus
E^{\prime\prime}\subset \Omega^{(0,1)} (\Delta, u_{\mathrm{op}}|_{\Delta}^*
TX) \times \Omega^{(0,1)} (\Sigma_1, u_{\mathrm{op}}|_{\Sigma_1}^* TX)$. But
since $D_{u_{\mathrm{cl}}|_{\Sigma_1}} \bar{\partial} = D_{u_{\mathrm{op}%
}|_{\Sigma_1}} \bar{\partial}$, $E^{\prime}$ and $E^{\prime\prime}$ can be
taken as the same subspace! Once we do this, it is then routine to see that $%
(V_{\mathrm{cl}},E_{\mathrm{cl}},\Gamma_{\mathrm{cl}},\psi_{\mathrm{cl}},s_{%
\mathrm{cl}}) = (V_{\mathrm{op}},E_{\mathrm{op}},\Gamma_{\mathrm{op}},\psi_{%
\mathrm{op}},s_{\mathrm{op}})$.
\end{proof}

\begin{thm}
\label{thm_main} Let $X_0$ be a toric Calabi-Yau threefold and denote by $S$
the union of its compact toric divisors. Let $L$ be a Lagrangian torus fiber
and $b = \beta + \alpha \in \pi_2(X_0,L)$, where $\alpha \in H_2(S)$ and $%
\beta \in \pi_2(X_0,L)$ is of Maslov index two with $\beta \cdot S = 1$.

Given this set of data, a toric Calabi-Yau threefold $W_0$ can be
constructed explicitly with the following properties:

\begin{enumerate}
\item $W_0$ is birational to $X_0$.

\item Let $S_1 \subset W_0$ be the union of compact divisors of $W_0$. Then $%
S_1$ is the blow up of $S$ at one point, with $\alpha^{\prime}\in H_2(S_1)$
being the strict transform of $\alpha \in H_2(S)$.
\end{enumerate}

\noindent Then the open Gromov-Witten invariant $n_b$ of $(X_0,L)$ is equal
to the ordinary Gromov-Witten invariant $\langle 1 \rangle^{W_0}_{0,0,
\alpha^{\prime}}$ of $W_0$, i.e., 
\begin{equation*}
n_b = \langle 1 \rangle^{W_0}_{0,0, \alpha^{\prime}} 
\end{equation*}
provided that every rational curve representative of $\alpha^{\prime}$ in $%
W_0$ lies in $S_1$.
\end{thm}

In particular, when $X_0=K_S$, we obtain Theorem \ref{cor_main} as its
corollary.

\begin{proof}
{We first construct the toric variety $W_0$. To begin with, let $D_\infty$
be the toric divisor corresponding to $v_\infty$. Let $x \in X$ be one of
the torus-fixed points contained in $D_\infty$. First we blow up $x$ to get $%
X_1$, whose fan $\Sigma_1$ is obtained by adding the ray generated by $w =
v_\infty + u_1 + u_2$ to $\Sigma$, where $v_\infty$, $u_1$ and $u_2$ are the
normal vectors to the three facets adjacent to $x$. There exists a unique
primitive vector $u_0 \not= w$ such that $\{ u_0, u_1, u_2 \}$ generates a
simplicial cone in $\Sigma_1$ and $u_0$ corresponds to a compact toric
divisor of $X_1$: If $\{v_0, u_1, u_2\}$ spans a cone of $\Sigma_1$, then
take $u_0 = v_0$; otherwise since $\Sigma_1$ is simplicial, there exists a
primitive vector $u_0 \subset \mathbb{R}\langle v_0, u_1, u_2 \rangle$ with
the required property. Now $\langle u_1, u_2, w \rangle_\bR$ and $\langle
u_1, u_2, u_0 \rangle_\bR$ form two adjacent simplicial cones in $\Sigma_1$,
and we may employ a flop to obtain a new toric variety $W$, whose fan $%
\Sigma_W$ contains the adjacent cones $\langle w, u_0, u_1 \rangle_\bR$ and $%
\langle w , u_0, u_2 \rangle_\bR$. (See Figure \ref{fig_flop}). }

\begin{figure}[h!]
\begin{center}
\setlength{\unitlength}{3mm} 
\begin{picture}(28,10)(0,0)
\linethickness{0.075mm}
\put(0,0){\line(1,0){6}}\put(6,0){\line(0,1){6}}\put(0,0){\line(0,1){6}}
\put(0,6){\line(1,0){6}}
\put(0,6){\line(1,-1){6}}\put(-1,-1){$u_0$}\put(6,-1){$u_2$}\put(-1,6.5){$u_1$}\put(6,6.5){$w$}
\put(20,0){\line(1,0){6}}\put(26,0){\line(0,1){6}}\put(20,0){\line(0,1){6}}
\put(20,6){\line(1,0){6}}
\put(20,0){\line(1,1){6}}\put(19,-1){$u_0$}\put(26,-1){$u_2$}\put(19,6.5){$u_1$}\put(26,6.5){$w$}
\end{picture}
\end{center}
\caption{A flop.}
\label{fig_flop}
\end{figure}

$W$ is the compactification of another toric Calabi-Yau $W_0$ whose fan is
constructed as follows: First we add the ray generated by $w$ to $\Sigma_0$,
and then we flop the adjacent cones $\langle w, u_1, u_2 \rangle$ and $%
\langle u_0, u_1, u_2 \rangle$. $W_0$ is Calabi-Yau because 
\begin{equation*}
\left( \underline{\nu}\, , \, w \right) = 1 
\end{equation*}
and a flop preserves this Calabi-Yau condition. $\Sigma_W$ is recovered by
adding the ray generated by $v_\infty$ to the fan $\Sigma_{W_0}$.

Now we analyze the transform of classes under the above construction. The
class $h \in H_2(X, \mathbb{Z})$ can be written as $h^{\prime}+ \delta$,
where $h^{\prime}\in H_2(X, \mathbb{Z})$ is the class corresponding to the
cone $\langle u_1, u_2 \rangle_\bR$ of $\Sigma$ and $\delta \in H_2 (X_0, 
\mathbb{Z})$. Let $h^{\prime\prime}\in H_2(X_1, \mathbb{Z})$ be the class
corresponding to $\{u_1, u_2\} \subset \Sigma_1$, which is flopped to $e \in
H_2(W, \mathbb{Z})$ corresponding to the cone $\langle w, u_0 \rangle_\bR$
of $\Sigma_W$. Finally let $\tilde{\delta}, \tilde{\alpha} \in H_2(W, 
\mathbb{Z})$ be classes corresponding to $\delta, \alpha \in H_2(X_1, 
\mathbb{Z})$ respectively under the flop. Then $\alpha^{\prime}=\tilde{\delta%
} + \tilde{\alpha} - e$ is actually the strict transform of $\alpha$.

Applying Proposition \ref{openclose2} and Theorem \ref{main}, we obtain the
equality 
\begin{equation*}
n_b = \langle 1 \rangle^{W_0}_{0,0, \alpha^{\prime}}.
\end{equation*}
\end{proof}

Finally we give an example to illustrate the open Gromov-Witten invariants.

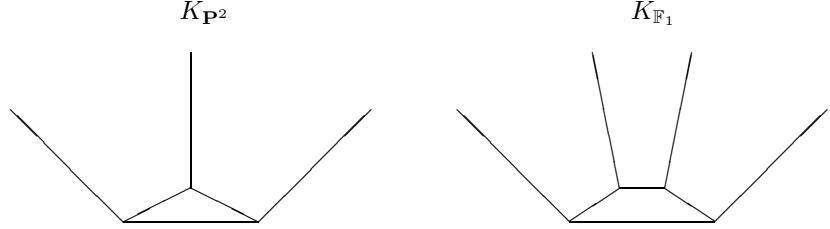
\begin{figure}[h!]
\begin{center}
\setlength{\unitlength}{3mm} 
\begin{picture}(28,10)(0,0)
\linethickness{0.075mm}
\put(0,0){\line(1,0){6}}\put(6,0){\line(1,1){5}}\put(0,0){\line(-1,1){5}}
\put(3,1.5){\line(0,1){6}}\put(3,1.5){\line(2,-1){3}}\put(3,1.5){\line(-2,-1){3}}
\put(2.5,9){$K_{\proj^2}$}
\put(19.8,0){\line(1,0){6.4}}\put(26.2,0){\line(1,1){5}}\put(19.8,0){\line(-1,1){5}}
\put(22,1.5){\line(-1,5){1.2}}\put(22,1.5){\line(1,0){2}}\put(22,1.5){\line(-3,-2){2.2}}
\put(24,1.5){\line(1,5){1.2}}\put(24,1.5){\line(3,-2){2.2}}
\put(22.5,9){$K_{\bF_1}$}
\end{picture}
\end{center}
\caption{Polytope picture for $K_{{\mathbf{P}}^2}$ and $K_{\mathbb{F}_1}$.}
\label{KP2_KF1}
\end{figure}

\begin{eg}
Let $X_0 = K_{{\mathbf{P}}^2}$. There is exactly one compact toric divisor $%
D_0$ which is the zero section of $X_0\to{\mathbf{P}}^2$. The above
construction gives $W_0 = K_{\mathbb{F}_1}$. (Figure \ref{KP2_KF1}). Let $%
\alpha = kl \in H_2(X_0, \mathbb{Z})$, where $l$ is the line class of ${%
\mathbf{P}}^2 \subset K_{{\mathbf{P}}^2}$ and $k > 0$. By Theorem \ref%
{thm_main}, 
\begin{equation*}
n_{\beta_0+kl} = \langle 1 \rangle^{W_0}_{0,0, kl - e}
\end{equation*}
where $e$ is the exceptional divisor of $\mathbb{F}_1 \subset K_{\mathbb{F}%
_1}$. The first few values of these local invariants for $K_{\mathbb{F}_1}$
are listed in Table \ref{tableF1}.
\end{eg}

\section{A generalization to ${\mathbf{P}^{n}}$-bundles}

\label{generalization}

In this section we generalize Theorem \ref{main} to higher dimensions, that
is, to ${\mathbf{P}^{n}}$-bundles over an arbitrary smooth projective
variety.

Let $X$ be an $n$-dimensional smooth projective variety. Let $F$ be a rank $r
$ vector bundle over $X$ with $1\le r<n$. Let $p:W=\mathbf{P}(F\oplus%
\mathcal{O}_X)\to X$ be a ${\mathbf{P}^{r}}$-bundle over $X$. There are two
canonical subvarieties of $W$, say $W_0=\mathbf{P}(0\oplus\mathcal{O}_X)$
and $W_\infty=\mathbf{P}(F\oplus 0)$. We have $W_0\cong X$.

Let $S\subset X$ be a smooth closed subvariety of codimension $r+1$ with
normal bundle $N$. Let $\pi:{\tilde X}\to X$ be the blowup of $X$ along $S$
with exceptional divisor $E=\mathbf{P}(N)$. Then $F^{\prime}=\pi^*F\otimes%
\mathcal{O}_{{\tilde X}}(E)$ is a vector bundle of rank $r$ over ${\tilde X}$%
. Similar to $p:W\to X$, we let $p^{\prime}:W^{\prime}=\mathbf{P}%
(F^{\prime}\oplus\mathcal{O}_{{\tilde X}})\to {\tilde X}$.

It is easy to see that $W$ and $W^{\prime}$ are birational. We shall
construct an explicit birational map $g:W\dashrightarrow W^{\prime}$. It
induces a homomorphism between groups 
\begin{equation*}
g^{\prime}:H_2(W,\mathbb{Z})\to H_2(W^{\prime},\mathbb{Z}).
\end{equation*}
Let $\beta=h+\alpha\in H_2(W,\mathbb{Z})$ with $h$ the fiber class of $W$
and $\alpha\in H_2(X,\mathbb{Z})$. Then we establish a relation between
certain Gromov-Witten invariants of $W$ and $W^{\prime}$.

\begin{prop}
\label{mainprop} {Let $Y=\mathbf{P}(F_S\oplus 0)\subset W$. For $%
g:W\dashrightarrow W^{\prime}$, we have 
\begin{equation*}
\langle\gamma_1,\gamma_2,\cdots,\gamma_{m-1},PD([Y])\rangle^W_{0,m,\beta}=%
\langle\gamma^{\prime}_1,\cdots,\gamma^{\prime}_{m-1}\rangle^{W^{%
\prime}}_{0,m-1,\beta^{\prime}}.
\end{equation*}
Here $\gamma^{\prime}_i$ is the image of $\gamma_i$ under $H^*(W)\to
H^*(W^{\prime})$ and $\beta^{\prime}=g^{\prime}(\beta)$. }
\end{prop}

The birational map $g:W\dashrightarrow W^{\prime}$ we shall construct below
can be factored as 
\begin{equation*}
W\overset{\pi_1^{-1}}{\dashrightarrow }{\tilde W}\overset{f}{\dashrightarrow 
}W^{\prime}
\end{equation*}
Here $\pi_1:{\tilde W}\to W$ is a blowup along a subvariety $Y$. We assume
that every curve $C$ in class $\beta$ can be decomposed uniquely as $C=H\cup
C^{\prime}$ with $H$ a fiber and $C^{\prime}$ a curve in $X$. It follows
that the intersection of $C$ and $Y$ is at most one point. Under this
assumption we generalize Theorem \ref{thmGH} in a straightforward manner as
follows.

\begin{prop}
\label{propblowup} Let $E^{\prime}$ be the exceptional divisor of $\pi_1$.
Let $e$ be the line class in the fiber of $E^{\prime}\to Y$. Then we have 
\begin{equation*}
\langle\gamma_1,\gamma_2,\cdots,\gamma_{m-1},PD([Y])\rangle^W_{0,m,\beta}=
\langle\tilde\gamma_1,\cdots,\tilde\gamma_{m-1}\rangle^{{\tilde W}%
}_{0,m-1,\beta_1},
\end{equation*}
where $\tilde\gamma_i=\pi_1^*\gamma_i$ and $\beta_1=\pi^!(\beta)-e$.
\end{prop}

The proof of Proposition \ref{mainprop} is similar to that of Theorem \ref%
{main}.

\begin{proof}[Proof of Proposition \protect\ref{mainprop}]
{Since $g=f\pi_1^{-1}$, applying Proposition \ref{propblowup}, it suffices
to show 
\begin{equation*}
\langle\tilde\gamma_1,\cdots,\tilde\gamma_{m-1}\rangle^{{\tilde W}%
}_{0,m-1,\beta_1}=\langle\gamma^{\prime}_1,\cdots,\gamma^{\prime}_{m-1}%
\rangle^{W^{\prime}}_{0,m-1,\beta^{\prime}}
\end{equation*}
for the ordinary flop $f:{\tilde W}\dashrightarrow W^{\prime}$. }

Recall that Y-P.~Lee, H-W.~Lin and C-L.~Wang \cite{LLW} proved that for an
ordinary flop $f:M\dashrightarrow M_f$ of splitting type, the big quantum
cohomology rings of $M$ and $M_f$ are isomorphic. In particular, their
Gromov-Witten invariants for the corresponding classes are the same.
Therefore, the above identity follows.
\end{proof}

In the rest of the section we construct the birational map $%
g:W\dashrightarrow W^{\prime}$ in two equivalent ways.

Recall that $S\subset X$ is a subvariety. Let $p_S:Z=W\times_X S \to S$ be
the restriction of $p:W\to X$ to $S$. Then $Z=\mathbf{P}(F_S\oplus \mathcal{O%
}_S)$ with $F_S$ the restriction of $F$ to $S$. We denote $Y=Z\cap W_\infty=%
\mathbf{P}(F_S\oplus 0)$, and $q:Y\to S$ the restriction of $p_S$ to $Y$.
Since $Y$ is a projective bundle over $S$, we let $\mathcal{O}_{Y/S}(-1)$ be
the tautological line bundle over $Y$. The normal bundle of $Y$ in $Z$ is $%
N_{Y/Z}=\mathcal{O}_{Y/S}(1)$.

We start with the first construction of $g$. Let $\pi_1:{\tilde W}\to W$ be
the blowup of $W$ along $Y$. Since the normal bundle $N_{Y/W}$ is equal to $%
N_{Y/Z}\oplus N_{Y/W_\infty}=\mathcal{O}_{Y/S}(1)\oplus q^*N$, the
exceptional divisor of $\pi_1$ is 
\begin{equation*}
E^{\prime}=\mathbf{P}(\mathcal{O}_{Y/S}(1)\oplus q^*N).
\end{equation*}
Let ${\tilde Z}$ be the proper transform of $Z$ and ${\tilde Y}={\tilde Z}%
\cap E^{\prime}$. The normal bundle of ${\tilde Z}$ in ${\tilde W}$ is ${%
\tilde N}=p_S^*N\otimes\mathcal{O}_{{\tilde Z}}(-{\tilde Y})$.

Because $Z^{\prime}\cong Z$ is a ${\mathbf{P}^{r}}$-bundle over $S$, and the
restriction of ${\tilde N}$ to each ${\mathbf{P}^{r}}$-fiber of ${\tilde Z}$
is isomorphic to $\mathcal{O}(-1)^{\oplus r+1}$, we have an ordinary ${%
\mathbf{P}^{r}}$-flop $f:{\tilde W}\dashrightarrow {\tilde W}_f$ along ${%
\tilde Z}$. It can be verified that ${\tilde W}_f=W^{\prime}$ after
decomposing $f$ as a blowup and a blowdown. Finally we simply define $g$ as
the composite $f\pi_1^{-1}:W\dashrightarrow W^{\prime}$.

We describe the second construction of $g$, from which it is easy to see the
relation ${\tilde W}_f=W^{\prime}$.

We let $\rho_1:W_1\to W$ be the blowup of $W$ along $Z$ whose exceptional
divisor is denoted by $E_1$. Because the normal bundle of $Z$ in $W$ is $q^*N
$ for $q:Z\to S$, we know 
\begin{equation*}
E_1=\mathbf{P}(q^*N)\cong Z\times_S \mathbf{P}(N)=Z\times_S E.
\end{equation*}
Indeed, $W_1$ is isomorphic to the ${\mathbf{P}^{r}}$-bundle $\mathbf{P}%
(F_1\oplus\mathcal{O}_{{\tilde X}})$ over ${\tilde X}$ with $F_1=\pi^*F$.
Let $Y_1$ be the inverse image of $Y$. Now we let $\rho_2:W_2\to W_1$ be the
blowup of $W_1$ along $Y_1$ with exceptional divisor $E_2$. Let $E_1^{\prime}
$ be the proper transform of $E_1$ and $Y_2=E_1^{\prime}\cap E_2$. Notice
that $E_1^{\prime}\cong E_1$, and the normal bundle of $E_1$ is $%
N_1=q^*N\boxtimes\mathcal{O}_{E/S}(-1)$, we know the normal bundle of $%
E_1^{\prime}$ is $N_1^{\prime}=N_1\otimes\mathcal{O}_{E_1^{\prime}}(-Y_2)$.

Since $E_1^{\prime}\cong Z\times_S E$ is a ${\mathbf{P}^{r}}\times {\mathbf{P%
}^{r}}$-bundle over $S$, composed with the projection $Z\times_S E\to E$, we
see that $E_1^{\prime}\to E$ is a ${\mathbf{P}^{r}}$-bundle. Because the
restriction of $N_1^{\prime}$ to the ${\mathbf{P}^{r}}$-fiber of $%
E_1^{\prime}\to E$ is isomorphic to $\mathcal{O}(-1)^{\oplus r+1}$, we can
blowdown $W_2$ along these fibers of $E_1^{\prime}$ to get $\pi_3:W_2\to W_3=%
{\tilde W}_f$. From this description it is easy to see that $W_3=W^{\prime}$.


\begin{thebibliography}{99}
\bibitem{A} D.~Auroux, \textsl{Special Lagrangian fibrations, wall-crossing,
and mirror symmetry.} Surveys in dif- ferential geometry. Vol. XIII.
Geometry, analysis, and algebraic geometry: forty years of the Journal of
Differential Geometry, 1--47, Surv. Differ. Geom., 13, Int. Press,
Somerville, MA, 2009. arXiv:0902.1595.

\bibitem{A2} D.~Auroux, \textsl{Mirror symmetry and $T$-duality in the
complement of an anticanonical divisor.} J. G\"okova Geom. Topol. GGT 1
(2007), 51--91.

\bibitem{C} K.-W.~Chan, \textsl{Mirror symmetry for a class of toric nef
manifolds.} arXiv:1006.3827.

\bibitem{CLL} K.-W.~Chan, S.-C.~Lau and N.C.~Leung, \textsl{SYZ mirror
symmetry for toric Calabi-Yau manifolds.} arXiv:1006.3830.

\bibitem{CL} K.-W.~Chan and N.C.~Leung, \textsl{Mirror symmetry for toric
Fano manifolds via SYZ transformations.} Adv. Math. 223 (2010), no. 3,
797--839.

\bibitem{CKYZ} T.-M.~Chiang, A.~Klemm, S.-T.~Yau and E.~Zaslow, \textsl{%
Local Mirror Symmetry: Calculations and Interpretations.} Adv. Theor. Math.
Phys. 3 (1999), no. 3, 495--565.

\bibitem{Cho-Oh} C.-H.~Cho and Y.-G.~Oh, \textsl{Floer cohomology and disc
instantons of Lagrangian torus fibers in Fano toric manifolds.} Asian J.
Math. 10 (2006), no. 4, 773--814.

\bibitem{FOOO} K.~Fukaya, Y.-G.~Oh, H.~Ohta and K.~Ono, \textsl{Lagrangian
Floer theory on compact toric manifolds I.} Duke Math. J. 151 (2010), no. 1,
23--174.

\bibitem{FOOO2} K.~Fukaya, Y.-G.~Oh, H.~Ohta and K.~Ono, \textsl{Toric
degeneration and non-displaceable Lagrangian tori in $S^2 \times S^2$.}
arXiv:1002.1660.

\bibitem{FOOO_book} K.~Fukaya, Y.-G.~Oh, H.~Ohta and K.~Ono, \textsl{%
Lagrangian intersection {F}loer theory: anomaly and obstruction.} AMS/IP
Stud. Adv. Math. 46, Amer. Math. Soc., Providence, 2009.

\bibitem{Fukaya-Ono} K.~Fukaya and K.~Ono, \textsl{Arnold conjecture and
Gromov-Witten invariant.} Topology 38 (1999), no. 5, 933--1048.

\bibitem{Ga} A.~Gathmann, \textsl{Gromov-Witten invariants of blow-ups.} J.
Algebraic Geom. 10 (2001), no. 3, 399--432.

\bibitem{Hu} J.~Hu, \textsl{Gromov-Witten invariants of blow-ups along
points and curves.} Math. Z. 233 (2000), no. 4, 709--739.

\bibitem{Hu2} J.~Hu, \textsl{Local Gromov-Witten invariants of blowups of
Fano surfaces.} arXiv:1006.4233.

\bibitem{L} N.C.~Leung, \textsl{Mirror symmetry without corrections.} Comm.
Anal. Geom. 13 (2005), no. 2, 287--331.

\bibitem{LLW} Y.-P.~Lee, H.-W.~Lin and C.-L.~Wang, \textsl{Flops, motives
and invariance of quantum rings.} To appear in Ann. of Math.
arXiv:math/0608370.

\bibitem{LR} A.-M.~Li and Y.~Ruan, \textsl{Symplectic surgery and
Gromov-Witten invariants of Calabi-Yau 3-folds.} Invent. Math. 145 (2001),
no. 1, 151--218.

\bibitem{LYZ} N.C.~Leung, S.-T.~Yau and E.~Zaslow, \textsl{From special
Lagrangian to Hermitian-Yang-Mills via Fourier-Mukai transform.} Adv. Theor.
Math. Phys. 4 (2000), no. 6, 1319--1341.

\bibitem{SYZ} A.~Strominger, S.-T.~Yau and E.~Zaslow, \textsl{Mirror
symmetry is $T$-duality.} Nuclear Phys. B 479 (1996), no. 1-2, 243--259.
\end{thebibliography}
\end{document}